\documentclass[12pt]{amsart}
\usepackage{amsmath,amssymb,amsfonts,latexsym}
\usepackage{mathrsfs}
\usepackage[dvips]{geometry}
\usepackage{array}
\DeclareFontFamily{T1}{calligra}{}
\DeclareFontShape{T1}{calligra}{m}{n} {<-> callig15}{}

\newtheorem*{prop}{Proposition}
\theoremstyle{definition}

\theoremstyle{remark}
\newtheorem*{conj}{Conjecture}

\numberwithin{equation}{subsection}

\newcommand{\B}{\mathbf{B}}

\newcommand{\Co}{\mathbb{C}}

\begin{document}
\title[A conjecture of Panyushev]{B-orbits in abelian nilradicals of types $B, C$ and $D$: towards a conjecture of Panyushev}
\author[N. Barnea]{Nurit Barnea}
\thanks{N. Barnea: partially supported by Israel Scientific Foundation grant 797/14}
\address{Department of Mathematics,
University of Haifa, Haifa 31905, Israel}
\email{barnea.nurit@gmail.com}
\author[A. Melnikov]{Anna Melnikov}
\address{Department of Mathematics,
University of Haifa, Haifa 31905, Israel}
\email{melnikov@math.haifa.ac.il}

\begin{abstract}
Let $B$ be a Borel subgroup of a semisimple algebraic group $G$ and let $\mathfrak m$ be an abelian nilradical in $\mathfrak b={\rm Lie} (B)$. Using subsets of strongly orthogonal roots in the subset of positive roots corresponding to $\mathfrak m$,
D. Panyushev \cite{Pan} gives in particular classification of $B-$orbits in $\mathfrak m$ and ${\mathfrak m}^*$ and states  general conjectures on the closure and dimensions of the $B-$orbits in both $\mathfrak m$ and ${\mathfrak m}^*$ in terms of involutions of the Weyl group. Using Pyasetskii correspondence between 
$B-$orbits in $\mathfrak m$ and ${\mathfrak m}^*$ he shows  the equivalence of these two  conjectures. In this Note we prove his conjecture in types $B_n, C_n$ and $D_n$ for adjoint case.
\end{abstract}
\maketitle

\section{Abelian nilradicals and Panyushev's conjecture}
\subsection{Minimal nilradicals}\label{1.1}
 Let $G$ be a semisimple linear algebraic group
over $\mathbb{C}$ and let $\mathfrak{g}$ be its Lie algebra. Let $B$ be its Borel subgroup and  $\mathfrak b={\rm Lie} (B)$. Let
$\mathfrak g=\mathfrak n\oplus\mathfrak h\oplus\mathfrak n^-$ be its corresponding triangular decomposition, where $\mathfrak b=\mathfrak n\oplus\mathfrak h.$ $B$ acts adjointly on $\mathfrak n$. For $x\in\mathfrak n$ let $B.x$ denote its orbit.

Since the description of $B-$orbits in $\mathfrak n$ immediately reduces to simple Lie algebras in what follows we assume that $\mathfrak g$ is simple.

Let $R$ be the root system of $\mathfrak g$ and $W$ its Weyl group. For $\alpha\in R$ let $s_\alpha$ be the corresponding reflection in $W$. 

Let $R^+$ (resp. $R^-$) denote the subset of positive (resp. negative) roots. For $\alpha\in R$ let $X_\alpha$ denote the standard root vector in $\mathfrak g$  so that $\mathfrak n=\bigoplus\limits_{\alpha\in R^+}{\mathbb C} X_\alpha.$
Let  $\Delta =\{\alpha_i\}_{i=1}^n\subset R^+$ be a set of simple roots. Let $\theta$ be the maximal root in $R^+.$ 

Recall that any standard parabolic subgroup $P$ of $G$ is of the form $P=L\ltimes M$ where $L$ is a standard Levy subgroup and $M$ is the unipotent radical of $P.$ If $R_L$ is the root system of $\mathfrak l={\rm Lie}(L)$ then $\Delta_L=\Delta\cap R_L$. Let $W_P$ denote Weyl group of $\mathfrak l.$ Let $\widehat w$ be the longest element of $W_P$.

 $P$ is maximal if and only if $\Delta_L=\Delta\setminus\{\alpha_i\}$. We will write $P_{\alpha_i}=P$, $M_{\alpha_i}=M$, $R_{\alpha_i}=R_L$, $R^+_{\alpha_i}=R_L^+$ and $W_{\alpha_i}=W_P$ in this case. We put ${\overline R}_{\alpha_i}^+:=R^+\setminus R^+_{\alpha_i}$.   Put $\mathfrak m_{\alpha_i}:={\rm Lie}(M_{\alpha_i})=\bigoplus\limits_{
\alpha\in {\overline R}_{\alpha_i}^+}{\mathbb C} X_\alpha.$

A nilradical $\mathfrak m$ is abelian if and only if $\mathfrak m=\mathfrak m_{\alpha_i}$  and in $\theta=\sum\limits_{j=1}^n k_j\alpha_j$ one has $k_i=1$ 
(cf. \cite{Pan} for details). 
\subsection{Strongly orthogonal sets and $B-$orbits in $A_n,B_n,C_n,D_n$}\label{1.2}
A set $\mathcal S\subset R^+$ is called {\it strongly orthogonal} if $\alpha\pm \beta\not\in R$ for any $\alpha, \beta\in \mathcal S$. 
Given a strongly orthogonal set $\mathcal S=\{\beta_i\}_{i=1}^k$ put $\sigma_{\mathcal S}:=\prod\limits_{i=1}^k s_{\beta_i}$. Note that this is an involution.
As it is shown in \cite{Pan} each $B-$orbit in an abelian nilradical $\mathfrak m_{\alpha_i}$ has a unique representative of form $\sum\limits_{\alpha\in \mathcal S}X_\alpha$ where $\mathcal S\subset\overline{R}_{\alpha_i}^+$ is strongly orthogonal.

We choose the following root systems: 
\begin{itemize}
\item{} In $A_n:\ R=\{e_j-e_i\}_{1\leq i\ne j \leq n+1},\ R^+=\{e_j-e_i\}_{1\leq i<j \leq n+1},\ \Delta=\{e_{i+1}-e_i\}_{i=1}^{n}$;
\item{} In $C_n:\ R=\{\pm (e_j\pm e_i)\}_{1\leq i<j \leq n}\cup\{\pm 2e_i\}_{i=1}^n,\  R^+=\{e_j\pm e_i\}_{1\leq i<j \leq n}\cup\{ 2e_i\}_{i=1}^n,$ $\Delta=\{2e_1,e_{i+1}-e_i\}_{i=1}^{n-1}$;
\item{} In $B_n:\ R=\{\pm (e_j\pm e_i)\}_{1\leq i<j \leq n}\cup\{\pm e_i\}_{i=1}^n,\ R^+=\{e_j\pm e_i\}_{1\leq i<j \leq n}\cup\{ e_i\}_{i=1}^n,$
$\Delta=\{e_1,e_{i+1}-e_i\}_{i=1}^{n-1}$;
\item{} In $D_n:\ R=\{\pm (e_j\pm e_i)\}_{1\leq i<j \leq n},\ R^+=\{e_j\pm e_i\}_{1\leq i<j \leq n},$\\ $\Delta=\{e_2+e_1,e_{i+1}-e_i\}_{i=1}^{n-1}$.
\end{itemize} 
We call roots $\alpha=e_j\pm e_i$ or $\alpha=e_i (2e_i)$, $\beta=e_l\pm e_k$ or $\beta=e_k (2e_k)$ {\it disjoint} if $\{i,j\}\cap\{k,l\}=\emptyset.$

In $A_n$ and $C_n$ the roots $\alpha,\beta$ are strongly orthogonal iff they are disjoint. In these two cases, (as well as in $D_n$) root vector $X_{\alpha}$ is of nilpotency order two. As it is shown in \cite{Mel, B-M} in theses two cases each $B-$orbit of nipotency order two in $\mathfrak n$ has a unique representative of the form
$\sum\limits_{i=1}^k X_{\beta_i}$ where $\{\beta_i\}_{i=1}^k\subset R^+$ is a strongly orthogonal (i.e. pairwise disjoint) set. On the other hand, each involution of $W$ can be written as a (commutative) product of pairwise disjoint reflections in the unique way, so   there is a one-to-one correspondence between the the strongly orthogonal sets  and involutions of $W$ so that $B-$orbits of nilpotency order 2 are indexed by involutions in these two cases. 

As for the cases $B_n$ and $D_n$ there is no bijection between $B-$orbits of nilpotent order 2 in $\mathfrak n$ and involutions of $W$ because of two reasons.
 First of all,a  root vector  $X_{e_i}$ in $B_n$ and a sum of strongly orthogonal root vectors $X_{e_j-e_i}+ X_{e_j+e_i}$ (roots $e_j-e_i$ and $e_j+e_i$ are strongly orthogonal in $\mathfrak{so}_n$) are matrices of nilpotency order 3 both in $B_n$ and $D_n$. The second obstacle  is that different sets of strongly orthogonal roots correspond to the same involution in $W,$ for example, $\sigma_{\{e_j-e_i,e_j+e_i\}}=\sigma_{\{e_i,e_j\}}$ but $X_{e_j-e_i}+ X_{e_j+e_i}$ and $X_{e_i}+X_{e_j}$ are representatives of different $B-$orbits (of nilpotency order 3) in $B_n$. Exactly in the same way $\sigma_{\{e_i,e_j,e_k,e_l\}}$ is connected to 3 different strongly orthogonal sets in $D_n$ namely $\{e_{t_1}-e_{s_1}, e_{t_1}+e_{s_1}, e_{t_2}-e_{s_2}, e_{t_2}+e_{s_2}\}$
where $\{s_1,s_2,t_1,t_2\}=\{i,j,k,l\}$ and $s_r<t_r$ for $r=1,2$ (and additional 7 different strongly orthogonal sets in $B_n$) and the corresponding sums of roots are representatives of different $B-$orbits (of nilpotency order 3). However when we restrict ourselves to abelian nilradicals there is a bijection between the sets of strongly orthogonal roots in $\overline{R}^+_{\alpha_i}$ and subset of involutions of $W$ so that $B-$orbits are indexed by involutions inside abelian nilradicals in the unique way. Some of these orbits are of nilpotency order 3. 
\subsection{Abelian nilradicals in $A_n,B_n,C_n,D_n$}\label{1.3}
Abelian nilradicals in $A_n,B_n,C_n,D_n$ are (cf. \cite{Pan}, for example for the details).
 
\begin{itemize}
\item[(i)] In ${\mathfrak{sl}}_n$ any $\mathfrak m_{e_{k+1}-e_k}$ is abelian so that there are $n-1$ abelian nilradicals. They are of the form
$${\mathfrak m}_{e_{k+1}-e_k}=\bigoplus\limits_{1\leq i\leq k<j\leq n}\Co X_{e_j-e_i}$$
 One can see at once that in this  case  ${\mathfrak m}_{e_{k+1}-e_k}$ is a subspace of matrices of nilpotency order 2 and respectively all $B-$orbits there are indexed by  sets of pairwise disjoint roots $\{e_{j_s}-e_{i_s}\}_{s=1}^m$ where $i_s\leq k$ and $j_s\geq k+1$ for any $s\ :\ 1\leq s\leq m$. 

\item[(ii)]\ In $\mathfrak{sp}_{2n}$ the abelian nilradical is unique and it is 
$$\mathfrak m_{2e_1}=\bigoplus\limits_{1\leq i<j\leq n}\Co X_{e_j+e_i}\oplus\bigoplus\limits_{i=1}^n\Co X_{2e_i}.$$
Again this is a subspace of matrices of nilpotency order 2, so that all the $B-$orbits there are indexed by sets of pairwise disjoint roots $\{2e_{i_s}\}_{s=1}^l\cup\{e_{k_t}+e_{j_t}\}_{t=1}^m$.
\item[(iii)]\ \ In $\mathfrak{so}_{2n+1}$ the abelian nilradical is unique and it is
$$\mathfrak m_{e_n-e_{n-1}}=\bigoplus\limits_{i=1}^{n-1}\Co X_{e_n-e_i}\oplus\bigoplus\limits_{i=1}^{n-1}\Co X_{e_n+e_i}\oplus\Co X_{e_n}.$$
By \cite{Pan}   $\{X_{e_n},X_{e_n\pm e_i},X_{e_n- e_i}+X_{e_n+ e_i}\}_{i=1}^{n-1}$ is the set of the (unique) representatives of $B-$orbits in the form of sums of strongly orthogonal root vectors. Note that the corresponding set of involutions $\{s_{e_n}, s_{e_n\pm e_i}, s_{e_n}s_{e_i}\}_{i=1}^{n-1}$ is defined uniquely on this subset.
\item[(iv)]\ In $\mathfrak{so}_{2n}$ there are 3 abelian nilradicals; two of them are isomorphic, namely,  $\mathfrak m_{e_2-e_1}\cong\mathfrak m_{e_2+e_1}$. It is enough to consider 
$$\mathfrak m_{e_2+e_1}=\bigoplus\limits_{1\leq i<j\leq n}\Co X_{e_j+e_i}$$
This is the subspace of matrices of nilpotency order 2 and a $B-$orbit in it has a unique representative in the form $\sum_{s=1}^m X_{e_{j_s}+e_{i_s}}$ where $\{e_{j_s}+e_{i_s}\}_{s=1}^m$ is a set of pairwise disjoint roots.

The third nilradical is
$$\mathfrak m_{e_n-e_{n-1}}=\bigoplus\limits_{i=1}^{n-1}X_{e_n-e_i}\oplus\bigoplus\limits_{i=1}^{n-1}X_{e_n+e_i}.$$
By \cite{Pan}  $\{X_{e_n\pm e_i},\  X_{e_n- e_i}+X_{e_n+ e_i}\}_{i=1}^{n-1}$ is the set of the (unique) representatives of $B-$orbits in the form of sums of strongly orthogonal root vectors. Note that the corresponding set of involutions $\{s_{e_n\pm e_i}, s_{e_n}s_{e_i}\}_{i=1}^{n-1}$ is defined uniquely on this subset.
\end{itemize}
In particular, as we see, all $B-$orbits in an abelian nilradical for $A_n,B_n,C_n$ and $D_n$ are indexed by strongly orthogonal subsets in $\overline{R}^+_{\alpha_i}$. For a strongly orthogonal set $\mathcal S\subset \overline{R}^+_{\alpha_i}$ put $\B_{\mathcal S}:=B.(\sum\limits_{\alpha\in \mathcal S}X_\alpha)$. 
\subsection{Panyushev's conjecture}\label{1.4}
To formulate the conjecture we need the following notation.
For $w\in W$ put $\ell(w)$ to be its length, that is $\ell(w):=\#\{\alpha\in R^+\ :\ w(\alpha)\in R^-\}$. For a strongly orthogonal set $\mathcal S$ let $\#(\mathcal S)$ denote its cardinality. Let $\leq$ denote Bruhat order on $W.$

Respectively, for (coadjoint) $B-$orbits in $\mathfrak m_\alpha^*$ Panyushev shows that they are labeled by the same strongly orthogonal sets $\mathcal S$ and we denote them by $\B^*_{\mathcal S}$.
\begin{conj} (Panyushev) Let $\mathfrak m_\alpha$ be an abelian nilradical in a simple $\mathfrak  g$, and $W_\alpha$ be the corresponding Weyl group. Let $\widehat w$ denote the longest element of $W_{\alpha}$.

 Let $\mathcal S, \mathcal S'\subset \overline{R}^+_\alpha$ be strongly orthogonal and let $\sigma=\sigma_{\mathcal S},\sigma'=\sigma_{\mathcal S'}$ . Then
\begin{itemize}
\item[i)] $\B_{\mathcal S}\subset\overline\B_{\mathcal S'}$ if and only if $\widehat w\sigma\widehat w\leq \widehat w\sigma'\widehat w$.
\item[ii)] $\dim B_{\mathcal S}=\frac{\ell(\widehat w\sigma\widehat w)+\#(\mathcal S)}{2}$;
\end{itemize}
Respectively, for coadjoint orbits one has
\begin{itemize}
\item[$i^*$)] $\B^*_{\mathcal S}\subset\overline{\B^*}_{\mathcal S'}$ if and only if $\sigma\leq \sigma'$.
\item[$ii^*$)]\ $\dim \B^*_{\mathcal S}=\frac{\ell(\sigma)+\#(\mathcal S)}{2}$;
\end{itemize} 
\end{conj}
Panyushev shows, using Pyasetskii correspondence that these two conjectures are equivalent. 

Taking into account that by \cite{T} for $B-$orbits $\mathcal B,\mathcal B'$ one has $\mathcal B'$ is in the boundary of $\mathcal B$ iff ${\rm codim}_{\overline{\mathcal B}}\mathcal B'=1$ the part $(ii)$ of the conjecture follows straightforwardly from part $(i)$.

In cases of $A_n$ and $C_n$ both adjoint and coadjoint $B-$orbits of nilpotency order 2 are indexed by involutions \cite{B-M, Ig, Ig2, Mel} and by \cite{Ig, Ig2}  for involutions $\sigma,\sigma'\in W$ one has $\B^*_\sigma\subset\overline{\B^*}_{\sigma'}$ if and only if $\sigma\leq \sigma'$ so that the conjecture is a private case of a more general phenomenon.  

As for $B_n$ and $D_n$ we were informed by M. Ignatyev that  general description of inclusions of coadjoint $B-$orbit closures of nilpotent order 2  is not given by restriction of Bruhat order to involutions. We think that this happens because of the same difficulties with bijection between the strongly orthogonal sets and involutions that are described above.

For adjoint orbits in $A_n$ and $C_n$, in general, the combinatorial order on involutions defined by the inclusion of $B-$orbit closures of nilpotency order 2 is not connected to Bruhat order. However, for $B-$orbits in an abelian nilradical the conjecture is obtained as a straightforward corollary of  \cite{Mel, B-M}.
  
In this Note we reprove the conjecture for $A_n$ and $C_n$ and prove it for $B_n$ and $D_n$ for adjoint case. We also provide a simple combinatorial expression for $\ell(\sigma)$ for involutions in $S_n,W_{C_n}$ and $W_{D_n}$. To do this we introduce link patterns.  May be the expression can be obtained from the results of F. Incitti and is known to experts, but we have not found this result in the literature.
\section {Link patterns and $\ell(\sigma)$ for the Weyl group}\label{2}
 Recall that Weyl group of $\mathfrak{sl}_n$ is $S_n$ and its action on roots is obtained by extending linearly  $w(e_i)=e_{w(i)}$. 
 Weyl group $W_{C_n}$ of either $\mathfrak{sp}_{2n}$ or $\mathfrak{so}_{2n+1}$ is a group of  maps from $\{-n,\ldots,-1,1,\ldots,n\}$ onto itself symmetric around zero, namely    $i\mapsto j \ \Leftrightarrow\ -i\mapsto -j$ and its action  on roots is obtained by extending linearly $w(e_i)={\rm sign}(w(i))e_{|w(i)|}$.
Finally,  Weyl group $W_{D_n}$ is a subgroup of $W_{C_n}$  of maps sending even number of positive numbers to negative numbers. It acts on roots exactly in the same way as $W_{B_n}$. 

A {\it link pattern} on $n$ points with $k$ arcs is a graph  on $n$ (numbered) vertexes (drawn on a horizontal line) with $k$  disjoint edges $\{(i_s,j_s)\}_{s=1}^k$  (that is, $\{i_s,j_s\}\cap \{i_t,j_t\}=\emptyset$ for $1\leq s\ne t\leq k$) drawn over the line and called arcs. Vertex $f\not\in\{i_s,j_s\}_{s=1}^k$ is called a 
 {\it fixed point}.

 A strongly orthogonal set $\{e_{j_s}-e_{i_s}\}_{s=1}^k$ in $\mathfrak{sl}_n$ (or corresponding involution in  $S_n$) can be drawn as a link pattern on $1,\ldots,n$ with edges$\{i_s,j_s\}_{s=1}^k$; respectively a strongly orthogonal set in $C_n$ (or corresponding involution in $W_{C_n}$) can be drawn as a  link pattern symmetric around zero on $-n,\ldots,-1,1,\ldots,n$ where $2e_i$ corresponds to arc $(-i,i)$ and  $e_j\pm e_i$ for $0<i<j\leq n$ corresponds to two arcs $(\mp i,j)$ and $(\pm i,-j)$. 
 Respectively, for an involution of $W_{C_n}$ to be an element of  $W_{D_n}$ we need the even number of cycles of type $(-i,i)$ so that it can be drawn as a link pattern on $-n,\ldots,-1,1,\ldots,n$ symmetric around zero  with even number of arcs over zero.

 Given a strongly orthogonal set $\mathcal S=\{e_{j_s}- e_{i_s}\}_{s=1}^m$ (resp. $\mathcal S=\{e_{j_s}\mp e_{i_s}\}_{s=1}^l\cup\{e_{2k_t}\}_{t=1}^m$) let  $P_{\mathcal S}$ be the corresponding link pattern.  Let $|\mathcal S|$ denote the number of arcs in $P_{\mathcal S}$.
 Note that in case of $\mathfrak{sl}_n$ one has $\#(\mathcal S)=|\mathcal S|$; in case of $C_n$ or $D_n$ one has $\#(\mathcal S)\leq |\mathcal S|\leq 2\#(\mathcal S)$ (depending on the roots).

 Let $(a_1,b_1)\ldots(a_m,b_m)$, where $m=|\mathcal S|$, be the list of arcs of $P_{\mathcal S}$ written in such a way that $a_i<b_i.$ We also need the following statistics on $P_{\mathcal S}$:
\begin{itemize}
\item[i)] set $c(a_s,b_s):=\#\{t\ :\ a_t<a_s <b_t<b_s\}$ to be the number of arcs crossing the given arc $(a_s,b_s)$ on the left and
$c(\mathcal S):=\sum\limits_{s=1}^m c(a_s,b_s)$ to be the total number of crosses;
\item[ii)] set $r(a_s,b_s):=\#\{t\ :\ a_t>b_s\}$ to be the number of arcs to the right of the given arc $(a_s,b_s)$ and $r(\mathcal S):=\sum\limits_{s=1}^m r(a_s,b_s)$ to be the total number of arcs to the right of some arc;
\item[iii)]\ set $b(a_s,b_s):=\#\{p\ :\ a_s<p<b_s\ {\rm and}\ p\not\in \{a_t,b_t\}_{t=1}^m\}$ to be the number of fixed points under the given arc (bridge) $(a_s,b_s)$; and $b(\mathcal S):=\sum\limits_{s=1}^mb(a_s,b_s)$ to be the total number of fixed points under the arcs, or in other words the total number of bridges over all fixed points.
\end{itemize}
For example, let  $\mathcal S=\{e_2-e_1,\ e_6+e_3,\ 2e_4\}$ in $C_6$, then 
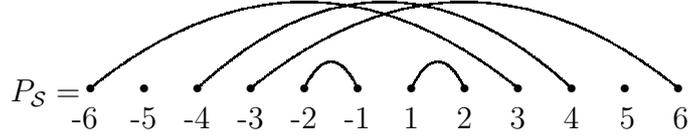
\begin{center}
\begin{picture}(130,60)(20,20)
\put(-40,35){$P_{\mathcal S}=$}
 \multiput(-10,40)(20,0){12} {\circle*{3}} 
 \put(-17,25){-6}
 \put(5,25){-5}
 \put(25,25){-4}
 \put(45,25){-3}
 \put(65,25){-2}
 \put(85,25){-1}
 \put(107,25){1}
 \put(127,25){2}
 \put(147,25){3}
 \put(167,25){4}
 \put(188,25){5}
 \put(208,25){6}
 \qbezier(-10,40)(70,105)(150,40)
 \qbezier(50,40)(130,105)(210,40)
 \qbezier(30,40)(100,105)(170,40)
 \qbezier(70,40)(80,60)(90,40)
 \qbezier(110,40)(120,60)(130,40)
\end{picture}
\end{center}
and $|\mathcal S|=5,$ $c(\mathcal S)=3,$ $r(\mathcal S)=1,$ $b(\mathcal S)=2.$ 
\begin{prop}Let $\mathcal S$ be a strongly orthogonal set in either $\mathfrak{sl}_n$ or $C_n$ ($D_n$) and let $\sigma=\sigma_{\mathcal S}$ be an involution in the corresponding Weyl group. 
\begin{enumerate} 
\item{For $\mathcal S=\{e_{j_s}-e_{i_s}\}_{s=1}^k$ in $\mathfrak{sl}_n$ one has 
$$\ell(\sigma)=2|\mathcal S|^2 -|\mathcal S|+2b(\mathcal S)-4r(\mathcal S)-2c(\mathcal S)$$}
\item{For $\mathcal S=\{e_{j_s}-e_{i_s}\}_{s=1}^a\cup\{2e_{k_s}\}_{s=1}^{d}\cup\{e_{m_s}+e_{l_s}\}_{s=1}^f$ in $C_n$  one has for $\sigma$ in $W_{C_n}$
 $$\ell(\sigma)=|\mathcal S|^2-a+b(\mathcal S)-c(\mathcal S)-2r(\mathcal S)$$}
 \item{For $\mathcal S=\{e_{j_s}-e_{i_s}\}_{s=1}^a\cup\{e_{k_s}\}_{s=1}^{2d}\cup\{e_{m_s}+e_{l_s}\}_{s=1}^f$ so that $\sigma\in W_{D_n}$ one has 
 $$\ell(\sigma)=|\mathcal S|^2 -|\mathcal S|+a+b(\mathcal S)-c(\mathcal S)-2r(\mathcal S)$$}
 \end{enumerate}
\end{prop}
\begin{proof}
We prove (1) by the induction on $|\sigma|$ and induction on $n$. It is trivial for $n=2.$ 
 Assume it is true for $\sigma\in S_{n-1}$ and show for $\sigma\in S_n$. Recall that $s_{e_j-e_i}=(i,j)$ in cyclic form so that 
 $|\mathcal S|=1$ iff $\sigma_{\mathcal S}=(i,j)$. If $(i,j)\ne(1,n)$ we can regard it as an element of
$S_{n-1}$ so that $\ell((i,j))$ is obtained by induction. For $(1,n)$ one has that $(1,n)(e_t-e_s)$ is negative iff $t=n$ or $s=1$ so that
  $\ell((1,n))=2n-3$.
On the other hand $b((1,n))=n-2$ and $c((1,n))=r((1,n))=0$ so that the expression is satisfied.

Now assume this is true for $\sigma_{\mathcal S'}\in S_{n}$ where $|\mathcal S'|\leq k-1$ and show for $\sigma_{\mathcal S}$ of $|\mathcal S|=k.$
Let $\sigma=\sigma'(i,j)$ where $\sigma'=(i_1,j_1)\ldots(i_{k-1},j_{k-1})$ and $j>j_s$ for any $1\leq s\leq k-1.$
If $j<n$ we can regard $\sigma$ as an element of $S_{n-1}$ and the result is obtained by induction on $n.$ If $j=n$ one has
$$\sigma(e_t-e_s)=\sigma'(i,n)(e_t-e_s)=\left\{\begin{array}{lll}-(e_n-e_i)&{\rm if}\  (s,t)=(i,n)&(I)\\
                                                                   e_i-\sigma'(e_s) &{\rm if}\ t=n,\ s\ne i&(II)\\
                                                                   e_n-\sigma'(e_s)&{\rm if}\ t=i&(III)\\
                                                                   \sigma'(e_t)-e_n&{\rm if}\ s=i,\ t<n&(IV)\\
                                                                   \sigma'(e_t-e_s)&{\rm otherwise}&(V)\\
                                                                   \end{array}\right.$$
Take into account that
\begin{itemize}
\item{}$\sigma'(e_n-e_i)=e_n-e_i$ so that case (I) adds 1 to the length;
\item{} $\sigma'(e_n-e_s)=e_n-\sigma'(e_s)\in R^+$. On the other hand $e_i-\sigma'(e_s)\in R^-$ exactly for $n-1-i$ roots since for every $i<s<n$ either $\sigma'(s)=s$ or there exists $r<j$ such that $\sigma'(r)=s.$ Thus $(II)$ adds $(n-1-i)$ to the length;
\item{} $\sigma(e_i-e_s)=e_n-\sigma'(e_s)\in R^+$ always. On the other hand for any $(i_s,j_s)$ such that $i_s<i<j_s$ one has
$\sigma'(e_i-e_{i_s})=e_i-e_{j_s}\in R^-$ so that in case $(III)$
 we have to reduce $c(i,n)$ from the length;
 \item{} $\sigma(e_t-e_i)=\sigma'(e_t)-e_n\in R^-$ for all $t\ :\ i<t$ and $\sigma'(e_t-e_i)=\sigma'(e_t)-e_i\in R^-$ iff $t=j_s$ where $i_s<i<j_s<j.$
 Thus case $(IV)$ adds $n-1-i-c(i,n)$ to the length;
\item{} Case $(V)$ does not add anything to the length.
 \end{itemize}
 Summarizing, we get  $\ell(\sigma)=2(n-i)-1-2c(i,n)+\ell(\sigma').$

 Put $u(i,n):=\#\{t\ :\ i<i_t,j_t<n\}$ to be the number of arcs under $(i,n)$.
 One has:\\
 $c(\mathcal S)=c(\mathcal S')+c(i,n)$;\\
 $b(\mathcal S)=b(\mathcal S')-c(i,n)+(n-1-i)-c(i,n)-2u(i,n)=b(\mathcal S')+(n-1-i)-2c(i,n)-2u(i,n)$\\
 $r(\mathcal S)=r(\mathcal S')+ (k-1)-c(i,n)-u(i,n)$ since for any $(i_s,j_s)$ it is either to the left of $(i,n)$ or under $(i,n)$ or crosses it on the left.\\
Taking all this into account we get straightforwardly $\ell(\sigma)=2k^2-k+2b(\mathcal S)-2c(\mathcal S)-4r(\mathcal S)$ in accordance with the expression.

(2) Let $\mathcal S=\{e_{j_s}-e_{i_s}\}_{s=1}^a\cup\{2e_{k_s}\}_{s=1}^{d}\cup\{e_{m_s}+e_{l_s}\}_{s=1}^f$ in $C_n$
so that $|\mathcal S|=2a+d+2f$ and let $\sigma=\sigma_{\mathcal S}$. Taking into account that $s_{2e_i}=(-i,i)$, $s_{e_j\pm e_i}=(\mp i,j)(\pm i,-j)$ by (i) its length as an element of $S_{2n}$ is $\ell_{S_{2n}}(\sigma)=|\mathcal S|^2 -|\mathcal S|+2b(\mathcal S)-4r(\mathcal S)-2c(\mathcal S)$.
On the other hand, in $C_n$ all the short roots are sums (up to sign) of two roots in $\mathfrak{sl}_{2n}$. Let $x(\sigma)$ be the number of positive long roots $2e_s$ such that 
$\sigma(2e_s)\in R^-.$ Then $\ell(\sigma)=\frac{1}{2}(\ell_{S_{2n}}(\sigma)+x(\sigma)).$ Further, note that $s_{e_j-e_i}(2e_s)\in R^+$ always, 
$$s_{2e_k}(2e_s)=\left\{\begin{array}{ll} -2e_s&{\rm if}\ s=k;\\
                                         2e_s&{\rm otherwise.}\\
																				\end{array}\right.\quad{\rm and}\quad
s_{e_j+e_i}(2e_s)= \left\{\begin{array}{ll} -2e_s&{\rm if}\ s=i,j;\\
                                         2e_s&{\rm otherwise.}\\
																				\end{array}\right.$$
Thus $x(\sigma)=d+2f=|\mathcal S|-2a.$ Summarizing, we get																				
  $\ell(\sigma)=|\mathcal S|^2-a+b(\mathcal S)-c(\mathcal S)-2r(\mathcal S)$.

(3)  Finally let $\mathcal S=\{e_{j_s}-e_{i_s}\}_{s=1}^a\cup\{e_{k_s}\}_{s=1}^{2d}\cup\{e_{m_s}+e_{l_s}\}_{s=1}^f$ so that $|\mathcal S|=2a+2d+2f$
and $\sigma\in W_{D_n}$. By (2) its length as an element of $W_{B_n}$ is $\ell_{B_n}(\sigma)=|\mathcal S|^2-a+b(\mathcal S)-c(\mathcal S)-2r(\mathcal S)$. Let $x(\sigma)$ be the number of positive short roots $e_s$ such that 
$\sigma(e_s)\in R_{B_n}^-.$ Then $\ell(\sigma)=\ell_{B_n}(\sigma)-x(\sigma).$ As it is shown in (2) $x(\sigma)=|\mathcal S|-2a.$ By a straightforward computation we get expression (3)
which completes the proof.
\end{proof}
\section{The proof of Panyushev's conjecture}

\subsection{Case $\mathfrak{sl}_n$}\label{3.1}
It is known that the conjecture is true for $\mathfrak{sl}_n$ (cf. \cite{Pan}). The proof is straightforward and we provide it in short here since we use it in what follows.

Let $S_n$ denote a standard symmetric group and $S_{[i,j]}$ a symmetric group on the elements $i,i+1,\ldots,j$.  For a strongly orthogonal set $\mathcal S\subset R^+$ let $\pi_{i,j}(\mathcal S)=\mathcal S\cap\{e_l-e_k\}_{i\leq k<l\leq j}$. By \cite{Mel}, $\B_{\mathcal S'}\subset\overline{\B}_{\mathcal S}$ for $\mathcal S,\mathcal S'\subset R^+$ strongly orthogonal sets in $\mathfrak{sl}_n$  iff for any $i,j\ :\ 1\leq i<j\leq n$ one has $|\pi_{i,j}(\mathcal S')|\leq |\pi_{i,j}(\mathcal S)|.$ Moreover these inclusions are generated by elementary moves on link patterns defined as follows:
\begin{enumerate}
\item{Let $e_j-e_i\in \mathcal S$  and let $\mathcal S'$ be obtained from $\mathcal S$ by exclusion of this root. Then $\B_{\mathcal S'}\subset \overline{\B}_{\mathcal S}$;}
\item{ Let $e_j-e_i\in \mathcal S$ and let $k>j$ be a fixed point of $P_{\mathcal S}$. Let $\mathcal S'$ be obtained from $\mathcal S$ by changing $e_j-e_i$ to $e_k-e_i$, then $\B_{\mathcal S'}\subset \overline{\B}_{\mathcal S}$;}
 \item{ Let $e_j-e_i\in \mathcal S$ and let $k<i$ be a fixed point of $P_{\mathcal S}$. Let $\mathcal S'$ be obtained from $\mathcal S$ by changing $e_j-e_i$ to $e_j-e_k$, then $\B_{\mathcal S'}\subset \overline{\B}_{\mathcal S}$;} 
\item{ Let $e_l-e_i, e_k-e_j\in \mathcal S$ be such that $i<j<k<l$. Let $\mathcal S'$ be obtained from $\mathcal S$ by changing $e_l-e_i, e_k-e_j$ to 
$e_k-e_i, e_l-e_j$, then $\B_{\mathcal S'}\subset \overline{\B}_{\mathcal S}$.}
\item{ Let $e_j-e_i, e_l-e_k\in \mathcal S$ be such that $j<k$. Let $\mathcal S'$ be obtained from $\mathcal S$ by changing $e_j-e_i, e_l-e_k$ to 
$e_k-e_i, e_l-e_j$, then $\B_{\mathcal S'}\subset \overline{\B}_{\mathcal S}$;}
\end{enumerate} 

For $\mathfrak m_{e_{k+1}-e_k}$ one has $W_{e_{k+1}-e_k}=S_k\times S_{[k+1,n]}$ and $\widehat w=[k,\ldots,1,n,\ldots,k+1]$.

Note that $\B_{\mathcal S}\subset\mathfrak m_{e_{k+1}-e_k}$ iff $\mathcal S=\{e_{j_s}-e_{i_s}\}_{s=1,\ i_s\leq k,j_s\geq k+1}^m$, and therefore,
$\widehat w\sigma_{\mathcal S}\widehat w=\sigma_{\widehat {\mathcal S}}$ where $\widehat{\mathcal S}=\{e_{n+1-j_s}-e_{k+1-i_s}\}_{s=1}^m$.
 
Since on one hand inclusion of $B-$orbit closures is generated by elementary moves on link patterns  and on the other hand Bruhat order is generated by products by $(i,j)$  we have only to compare these two actions.
 
For $\mathcal S=\{e_{j_s}-e_{i_s}\}_{s=1}^m$ put $\langle \sigma_{\mathcal S}\rangle:=\{i_s,j_s\}_{s=1}^m$ to be the list of end points of $P_{\mathcal S}$.
 We have to take into account that the restriction of Bruhat order to involutions is generated by 
$\sigma<\sigma(i,j)$ only if $\{i,j\}\cap\langle\sigma\rangle=\emptyset$, otherwise we have to compare $\sigma$ and $(i,j)\sigma(i,j)$.

Let $\sigma=\sigma_{\mathcal S}$ where $\mathcal S\subset \overline{R}^+_{e_{k+1}-e_k}$. Note that  in order for $(i,j)\sigma$ in the first case (resp. $(i,j)\sigma(i,j)$ in the second case) to be $\sigma_{\mathcal S'}$ for $\mathcal S'\subset \overline{R}^+_{e_{k+1}-e_k}$ one needs to choose $i\leq k$ and $j\geq k+1$ (resp. either $i,j\leq k$ or $i,j\geq k+1$).

\begin{itemize}
\item[i)] $\sigma\rightarrow \sigma(i,j)$: Let $\mathcal S'=\{e_{j_s}-e_{i_s}\}_{s=1}^l$ and let $e_j-e_i$
 be strongly orthogonal to $\mathcal S'$ then $\mathcal S=\mathcal S'\cup \{e_j-e_i\}$ is strongly orthogonal so that  $\B_{\mathcal S'}\subset \overline{\B}_{\mathcal S}$ by (1) on one hand and on the other hand  $\widehat w(i,j)\sigma_{\mathcal S'}\widehat w=
(k+1-i,n+1-j)\sigma_{\widehat{\mathcal S'}}> \sigma_{\widehat{\mathcal S'}}$;
\item[ii)] $\sigma\rightarrow (i,j)\sigma(i,j)$ where either $i,j\leq k$ or $i,j\geq k+1$ and $|\{i,j\}\cap\langle \sigma\rangle|=1$: Let $\mathcal S=\{e_{j_1}-e_{i_1}\}\cup\mathcal T$ where $\mathcal T=\{e_{j_s}-e_{i_s}\}_{s=2}^m$. Let $j=i_1$ and $i\not\in\{i_s\}_{s=1}^m$ (resp. $i=j_1$ and $j\not\in \{j_s\}_{s=1}^m$).
Let $\mathcal S'=\{e_{j_1}-e_i\}\cup\mathcal T$ (resp. $\mathcal S'=\{e_{j}-e_{i_1}\}\cup\mathcal T$). Then on one hand by (2) $\B_{\mathcal S'}\subset\overline{\B}_{\mathcal S}$
iff $i<i_1$ (resp.  by (3) iff $j>j_1$). On the other hand $\widehat w\sigma_{\mathcal S'}\widehat w=(k+1-i,n+1-j_1)\sigma_{\widehat{\mathcal T}}$ (resp.
$\widehat w\sigma_{\mathcal S'}\widehat w=(k+1-i_1,n+1-j)\sigma_{\widehat{\mathcal T}}$) 
and  $\widehat w\sigma_{\mathcal S}\widehat w=(k+1-i_1,n+1-j_1)\sigma_{\widehat{\mathcal T}}$ so that $\widehat w\sigma_{\mathcal S'}\widehat w<\widehat w\sigma_{\mathcal S}\widehat w$ iff $i<i_1$ (resp. $j>j_1$).  
\item[iii)]\ $\sigma\rightarrow (i,j)\sigma(i,j)$ where $\{i,j\}\subset \langle \sigma\rangle:$ Let $\mathcal S=\{e_{j_1}-e_{i_1},e_{j_2}-e_{i_2}\}\cup \mathcal T$ where $i_1<i_2(\leq k)$ and $(i,j)=(i_1,i_2)$ (this is equal to action on $\sigma$ by $(i,j)=(j_1,j_2)$). Then $(i,j)\sigma(i,j)=\sigma_{\mathcal S'}$ where $\mathcal S'=\{e_{j_1}-e_{i_2},e_{j_2}-e_{i_1}\}\cup \mathcal T$. 
On one hand by (4)
$\B_{\mathcal S'}\subset \overline{\B}_{\mathcal S}$ iff $j_1>j_2$, on the other hand 
$\widehat w\sigma_{\mathcal S'}\widehat w =(k+1-i_1,n+1-j_2)(k+1-i_2,n+1-j_1)\sigma_{\widehat{\mathcal T}}$ and $\widehat w\sigma_{\mathcal S}\widehat w =(k+1-i_1,n+1-j_1)
(k+1-i_2,n+1-j_2)\sigma_{\widehat{\mathcal T}}$ so that  $\widehat w\sigma_{\mathcal S'}\widehat w<\widehat w\sigma_{\mathcal S}\widehat w$ iff $j_2<j_1.$
\end{itemize} 

\subsection{Case $\mathfrak{sp}_{2n}$}\label{3.2}
For $\mathfrak{sp}_{2n}$ the unique abelian nilradical is $\mathfrak m_{2e_1}$.  In this case $W_{2e_1}=S_n$ and $\widehat w=[n,\ldots,1]$.
One has $\B_{\mathcal S}\subset\mathfrak m_{2e_1}$ iff $\mathcal S=\{2e_{k_s}\}_{s=1}^d\cup\{e_{m_s}+e_{l_s}\}_{s=1}^f$.

 In this case the conjecture  is obtained as a straightforward corollary of the result for $\mathfrak{sl}_{2n}$ and the  following facts:
 \begin{enumerate}
 \item{A set of strongly orthogonal roots $\mathcal S$ in $C_n$ can be considered as a set $\widetilde{\mathcal S}$ of $|\mathcal S|$ strongly orthogonal roots in 
$\mathfrak {sl}_{2n}$. In these terms 
for $\B_{\mathcal S},\B_{\mathcal S'}\subset R^+$ in $\mathfrak{sp}_{2n}$  one has by \cite{B-M} $\B_{\mathcal S'}\subset\overline \B_{\mathcal S}$  iff they are restriction to 
$\mathfrak{sp}_{2n}$ of the orbits $\B'_{\widetilde{\mathcal S}}\, ,\B'_{\widetilde{\mathcal S'}}$ from $\mathfrak{sl}_{2n}$ such that  $\B'_{\widetilde{\mathcal S'}}\subset\overline \B'_{\widetilde{\mathcal S}}$;}
  \item{  $\mathfrak m_{2e_1}$ of $\mathfrak{sp}_{2n}$ is the restriction to $\mathfrak{sp}_{2n}$ of $\mathfrak m_{e_{n+1}-e_n}$ of $\mathfrak{sl}_{2n}$;}
  \item{ $\sigma,\sigma'\in W_{C_n}$ are elements of $S_{2n}$ and $\sigma'<\sigma$ in $W_{C_n}$ iff $\sigma'<\sigma$ in $S_{2n}$ -- this is shown for example in \cite[\S 4]{Pr};}
  \item{ $\widehat w\in W_{2e_1}$ is identified with the maximal element of $S_n\times S_{[n+1,2n]}.$}
  \end{enumerate}
  \subsection{Case $\mathfrak{so}_{2n+1}$}\label{3.3}
For $\mathfrak{so}_{2n+1}$ the unique abelian nilradical  is $\mathfrak{m}_{e_n-e_{n-1}}$. In this case $W_{e_n-e_{n-1}}=W_{B_{n-1}}$ and $\widehat w=s_{e_1}\ldots s_{e_{n-1}}$.

$\B_{\mathcal S}\subset\mathfrak m_{e_n-e_{n-1}}$  if either $\mathcal S=\{e_n\}$ or $\mathcal S=\{e_n\pm e_i\}$,  or  $\mathcal S=\{e_n-e_i,e_n+e_i\}$ for $1\leq i<n$.
Note that $\widehat ws_{e_n\pm e_i}\widehat w=s_{e_n\mp e_i}$, $\widehat ws_{e_n}\widehat w=s_{e_n}$
and $\widehat w\sigma_{\{e_n-e_i,e_n+e_i\}}\widehat w=\sigma_{\{e_n-e_i,e_n+e_i\}}.$

The restriction of Bruhat order to our set of involutions is as follows (cf. \cite{Pr}, for example):

$$\begin{array}{ll}
s_{e_n+e_{n-1}}>s_{e_n+e_{n-2}}>\ldots>s_{e_n+e_1}>s_{e_n-e_1}>\ldots>s_{e_n-e_{n-1}};\quad\ s_{e_n}>s_{e_n-e_1}& (i)\\
s_{e_n-e_{n-1}}s_{e_n+e_{n-1}}>\ldots>s_{e_n-e_1}s_{e_n+e_1}>s_{e_n};\quad\ s_{e_n-e_i}s_{e_n+e_i}>s_{e_n+e_i}& (ii)\\

\end{array}$$
Also $s_{e_n+e_i}$ and $s_{e_n}$ are incompatible for any $i<n$. 
As for inclusions of $B-$orbit closures one has
\begin{itemize}
\item[i)] In order to show $\B_{\{e_n-e_{i-1}\}}\subset\overline{\B}_{\{e_n-e_i\}}$ note that $Exp(a X_{e_i-e_{i-1}}).X_{e_n-e_i}=X_{e_n-e_i}-aX_{e_n-e_{i-1}}$ so that by torus action we get $X_{e_n-e_{i-1}}\in\overline{\B}_{\{e_n-e_i\}}$.
This corresponds to
$s_{e_n+e_i}>s_{e_n+e_{i-1}}$ for $i\ :\ 2\leq i\leq n-1.$\\
Let us show that $\B_{\{e_n+e_1\}}\subset \overline \B_{\{e_n-e_1\}}$. 
Indeed, 
$Exp(a X_{e_1}).X_{e_n-e_1}=X_{e_n-e_1}-a X_{e_n}-\frac{a^2}{2}X_{e_n+e_1}$.
Further by torus action we  get 
 $X_{e_n+e_1}\in\overline \B_{\{e_n-e_1\}}$. This corresponds to $s_{e_n+e_1}>s_{e_n-e_1}$.\\
To show $\B_{\{e_n+e_{i+1}\}}\subset\overline\B_{\{e_n+e_i\}}$ note that  
$Exp(a X_{e_{i+1}-e_i}).X_{e_n+e_i}=X_{e_n+e_i}+a X_{e_n+e_{i+1}}$ so that by torus action we get
  $X_{e_n+e_{i+1}}\in \overline{\B}_{\{e_n+e_i\}}$. This corresponds to $s_{e_n-e_i}>s_{e_n-e_{i+1}}$ for $i\ :\ 1\leq i\leq n-2$.

Exactly in the same way, $Exp(a X_{e_1}).X_{e_n}=X_{e_n}+a X_{e_n+e_1}$ and then by torus action we get $X_{e_n+e_1}\in\overline{\B}_{\{e_n\}}$ which corresponds to  $s_{e_n-e_1}<s_{e_n}$. 

Obviously $\B_{\{e_n-e_i\}}$ and $\B_{\{e_n\}}$ are incompatible.
\item[ii)] To show $\B_{\{e_n-e_i,e_n+e_i\}}\subset\overline \B_{\{e_n-e_j,e_n+e_j\}}$ for $1\leq i<j\leq n-1$ we note as before that
$Exp(a(X_{e_j-e_i}+X_{e_j+e_i})).(X_{e_n-e_j}+X_{e_n+e_j})=X_{e_n-e_j}+X_{e_n+e_j} -a(X_{e_n-e_i}+X_{e_n+e_i})$ and then by torus action  we get $X_{e_n-e_i}+X_{e_n+e_i}\in \overline \B_{\{e_n-e_j,e_n+e_j\}}$ which corresponds to $s_{e_n-e_j}s_{e_n+e_j}>s_{e_n-e_i}s_{e_n+e_i}$ for $1\leq i<j\leq n-1.$

To show $\B_{\{e_n\}}\subset\overline \B_{\{e_n-e_j,e_n+e_j\}}$ for $1\leq j\leq n-1$ note that
 $Exp(\sqrt{2} X_{e_j}).(X_{e_n-e_j}+X_{e_n+e_j})=X_{e_n-e_j}-\sqrt{2}X_{e_n}$
 and then by torus action we get $X_{e_n}\in \overline \B_{\{e_n-e_j,e_n+e_j\}}$ for any $1\leq j\leq n-1$ which corresponds to
$s_{e_n-e_j}s_{e_n+e_j}>s_{e_n}$ for $1\leq j\leq n-1.$

 Obviously by torus action we get $X_{e_n-e_i}\in\overline \B_{\{e_n-e_i,e_n+e_i\}}$ which provides \\
$s_{e_n-e_i}s_{e_n+e_i}>s_{e_n+e_i}$ for $1\leq i\leq n-1.$
\end{itemize}
\subsection{Case $\mathfrak{so}_{2n}$}
Recall that there are 3 abelian nilradicals in the case of $\mathfrak{so}_{2n}$ namely $\mathfrak m_{e_2-e_1}\cong\mathfrak m_{e_2+e_1}$ and $\mathfrak m_{e_n-e_{n-1}}$.

Let us start with $\mathfrak m_{e_n-e_{n-1}}$ which can be obtained from the previous case.
In this case $W_{e_n-e_{n-1}}=W_{D_{n-1}}$ and 
 $$\widehat w=\left\{\begin{array}{ll} s_{e_1}\ldots s_{e_{n-1}}&{\rm if}\ n=2k+1;\\
s_{e_2}\ldots s_{e_{n-1}}&{\rm if}\ n=2k;\\
\end{array}\right.$$
$\B_{\mathcal S}\subset\mathfrak m_{e_n-e_{n-1}}$  if either $\mathcal S=\{e_n\pm e_i\}$ or  $\mathcal S=\{e_n-e_i,e_n+e_i\}$ for $1\leq i<n$.
 Note that $\widehat ws_{e_n\pm e_i}\widehat w=s_{e_n\mp e_i}$ for $i>1$, $\widehat ws_{e_n\pm e_1}\widehat w=\left\{\begin{array}{ll}s_{e_n\pm e_1}&{\rm if}\ n=2k;\\
s_{e_n\mp e_1}&{\rm if}\ n=2k+1;\\ \end{array}\right.$
and $\widehat w s_{e_n-e_i}s_{e_n+e_i}\widehat w=s_{e_n-e_i}s_{e_n+e_i}.$
 The restriction of Bruhat order from $W_{B_n}$ to $W_{D_n}$
provides
$$\begin{array}{ll}
s_{e_n+e_{n-1}}>s_{e_n+e_{n-2}}>\ldots>s_{e_n+e_1},\quad s_{e_n-e_1}>s_{e_n-e_2}>\ldots>s_{e_n-e_{n-1}}, & (i)\\
\quad\quad s_{e_n+e_2}>s_{e_n-e_1},\ s_{e_n+e_1}>s_{e_n-e_2}& (ii)\\
s_{e_n-e_{n-1}}s_{e_n+e_{n-1}}>\ldots>s_{e_n-e_1}s_{e_n+e_1},\, s_{e_n-e_i}s_{e_n+e_i}>s_{e_n+ e_i},\, s_{e_n-e_1}s_{e_n+e_1}>s_{e_n\pm e_1} & (iii)\\
\end{array}$$
The only differences with $W_{B_n}$ are that $s_{e_n+e_1},\ s_{e_n-e_1}$ are incompatible (they are of the same length by Proposition 1) and $s_{e_n}\not\in W_{D_n}.$ As for inclusions of $B-$orbit closures we have to take into account that inclusions of $B-$orbit closures in $D_n$ implies the inclusions of corresponding $B-$orbit closures in $B_n$ so that we have to check only the corresponding cases from \ref{3.3}. We get:
\begin{itemize}
\item[i)+ii)]\ \ \ \ \ Exactly as in $\mathfrak{so}_{2n+1}$ one has $\B_{\{e_n-e_{i-1}\}}\subset\overline{\B}_{\{e_n-e_i\}}$ for $i\ :\ 2\leq i\leq n-1$  which corresponds to
$s_{e_n+e_{n-1}}>\ldots>s_{e_n+e_2}$ and $s_{e_n+e_2}>\left\{\begin{array}{ll}s_{e_n+e_1}&{\rm if}\ n=2k+1;\\
s_{e_n-e_1}&{\rm if}\ n=2k;\\ \end{array}\right.$.
Further note that
$Exp(a X_{e_2+e_1}).X_{e_n-e_2}=X_{e_n-e_2}-a X_{e_n+e_1}$ so that by torus action we get $X_{e_n+e_1}\in\overline{\B}_{e_n-e_2}$ which corresponds to $s_{e_n+e_2}>\left\{\begin{array}{ll}s_{e_n-e_1}&{\rm if}\ n=2k+1;\\
s_{e_n+e_1}&{\rm if}\ n=2k;\\ \end{array}\right.$\\
 Exactly as in $\mathfrak{so}_{2n+1}$ one has   $\B_{e_n+e_{i+1}}\subset \overline{\B}_{\{e_n+e_i\}}$ for $1\leq i\leq n-2$ . This corresponds to 
	$s_{e_n-e_i}>s_{e_n-e_{i+1}}$ for $i\ :\ 2\leq i\leq n-2$ and $s_{e_n-e_2}<\left\{\begin{array}{ll}s_{e_n-e_1}&{\rm if}\ n=2k+1;\\
                                    s_{e_n+e_1}&{\rm if}\ n=2k;\\ \end{array}\right.$.\\
Let us show that $\B_{\{e_n+e_2\}}\subset \overline{\B}_{\{e_n-e_1\}}$. Indeed, $Exp(a X_{e_2+e_1}).X_{e_n-e_1}=X_{e_n-e_1}+a X_{e_n+e_2}$ so that by torus action we get  $X_{e_n+e_2}\in\overline \B_{\{e_n-e_1\}}$. This corresponds  to 
$s_{e_n-e_2}<\left\{\begin{array}{ll}s_{e_n+e_1}&{\rm if}\ n=2k+1;\\
                                    s_{e_n-e_1}&{\rm if}\ n=2k;\\ \end{array}\right.$.\\
 To finish (i) and (ii) we have to check that $\B_{e_n+e_1}\not\subset\overline{\B}_{e_n-e_1}.$ This
is obtained straightforwardly from the fact $\dim \B_{e_n+e_1}=\dim\B_{e_n-e_1}=n-1$. 
\item[iii)]\ Exactly as in \ref{3.3} one has $\B_{\{e_n-e_i,\, e_n+e_i\}}\subset\overline \B_{\{e_n-e_j,e_n+e_j\}}$ for $1\leq i<j\leq n-1$. 
 Obviously by torus action  we get $X_{e_n-e_j}\in\overline \B_{\{e_n-e_j,e_n+e_j\}}$  and $X_{e_n\pm e_1}\in \overline \B_{\{e_n-e_1,e_n+e_1\}}$
so we get all the relations from (iii). 
\end{itemize}

Since  $\mathfrak m_{e_2+e_1}\cong \mathfrak m_{e_2-e_1}$ it is enough to consider $\mathfrak m_{e_2+e_1}.$ One has
$$\mathfrak m_{e_2+e_1}=\bigoplus\limits_{1\leq i<j\leq n}\Co X_{e_j+e_i}$$
Comparing $\mathfrak m_{e_2+e_1}$ with $\mathfrak m_{2e_1}$ of $\mathfrak{sp}_n$ one can see at once that root vectors here correspond (up to sign in the sum) to root vectors in $\mathfrak m_{2e_1}$ for short roots. In particular, this is a subspace of matrices of nilpotency order 2. 
The truth of the conjecture for $\mathfrak m_{e_2+e_1}$ is obtained from its truth for $\mathfrak m_{2e_1}$ by the following facts:
\begin{enumerate}
\item{ The sets of strongly orthogonal roots in $\mathfrak m_{e_2+e_1}$ coincide with the sets of strongly orthogonal short roots 
in $\mathfrak m_{2e_1}$}
\item{Only for root $\alpha=e_j-e_i$ the action of $Exp(a X_{\alpha})$ on roots $X_{e_k+e_s}$ can be non-trivial both in $C_n$ and $D_n$ and this action in both cases coincide up to sign, apart from case $Exp(a X_{e_j-e_i}).X_{e_j+e_i}=\left\{\begin{array}{ll}X_{e_j+e_i}+2a X_{2e_j}&{\rm in}\ C_n;\\
                                                                                                        X_{e_j+e_i}&{\rm in}\ D_n;\\ \end{array}\right.$ which is irrelevant here.
Thus, $X_{\mathcal S'}\in \overline{\B}_{\mathcal S}$ for strongly orthogonal sets $\mathcal S',\ \mathcal S$ in $\mathfrak m_{e_2+e_1}$
iff   	$X_{\mathcal S'}\in \overline{\B}_{\mathcal S}$ in $\mathfrak m_{2e_1}$.}																																												
\item{ $\widehat w$ of $W_{e_2+e_1}$ (in $W_{D_n}$) is equal to $\widehat w$ of $W_{2e_1}$ (in $W_{C_n}$).}
\item{ Bruhat order restricted to multiplication of reflections of strongly orthogonal roots of type $e_i+e_j$ coincides for $W_{C_n}$ and $W_{D_n}$.
(cf., for example \cite[\S 4]{Pr}).}
\end{enumerate}
\bigskip

\noindent
{\bf Acknowledgements:}\ 
We would like to thank Dmitri Panyushev for sharing  his paper with us and for discussions during this work.

\end{document}